\documentclass[12pt]{amsart}

\usepackage{lipsum,subeqnarray}
\usepackage{amsfonts}
\usepackage{graphicx}
\usepackage{amsmath}
\usepackage{mathtools}
\usepackage{amssymb}
\usepackage{enumerate}
\usepackage{slashed}
\usepackage{graphicx}
\usepackage{newlfont}
\usepackage{amsrefs}
\usepackage{comment}
\usepackage[normalem]{ulem}
\usepackage{mathtools}
\usepackage{accents}
\usepackage{leftidx}
\usepackage{bbm}

\numberwithin{equation}{section}

\theoremstyle{plain}
\newtheorem{theorem}{Theorem}[section]
\newtheorem{lemma}{Lemma}

\newtheorem{cor}{Corollary}

\newtheorem{proposition}[theorem]{Proposition}
\usepackage{xcolor}

\theoremstyle{definition}

\theoremstyle{remark}

\usepackage{stackengine}
\usepackage{subeqnarray} 

\stackMath
\newcommand\tenq[2][1]{%
 \def\useanchorwidth{T}%
  \ifnum#1>1%
    \stackunder[0pt]{\tenq[\numexpr#1-1\relax]{#2}}{\scriptscriptstyle\sim}%
  \else%
    \stackunder[1pt]{#2}{\scriptscriptstyle\sim}%
  \fi%
}

\newcommand{\overbar}[1]{\mkern 1.6mu\overline{\mkern-1.6mu#1\mkern-1.6mu}\mkern 1.6mu}

\DeclareMathOperator*{\esssup}{ess\,sup}

\begin{document}
\title[Quasi-Relativistic Navier-Stokes Equations]{The Global Existence of Solutions to a Quasi-Relativistic Incompressible Navier-Stokes Model}
 
\author{Jaroslaw S. Jaracz} 
\begin{abstract}
   We introduce a new modified Navier-Stokes model in $3$ dimensions by modifying the convection term in the ordinary Navier-Stokes equations. This is done by replacing the convective term $(\textbf{u}\cdot \nabla) \textbf{u}$ by $(\textbf{v}\cdot \nabla)\textbf{u}$ with $\textbf{v}=c\textbf{u}/\sqrt{c^2+|\textbf{u}|^2}$ where $c$ is the speed of light. Thus we have that $|\textbf{v}|\leq c$ and for $|\textbf{u}|\ll c$ we have $\textbf{v} \approx \textbf{u}$. Thus the solutions to this system should yield a good approximation to the solutions of the ordinary Navier-Stokes equations under physically reasonable conditions. The modification of the convective term is a natural progression of the work done in \cite{JaraczLee}.  The property that $|\textbf{v}|\leq c$ embodies the notion that in relativity matter can't travel faster than the speed of light, giving the model its name. We prove that there exists a strong solution   $\textbf{u} \in L^2(0, T; \textbf{H}^2) \cap L^{\infty}(0, T; \textbf{V})$ with  $\textbf{u}' \in L^2(0, T; \textbf{L}^2)$ to our system of equations on either a smooth bounded domain $U\subset \textbf{R}^3$ or the flat $3$-torus $\mathbb{T}$ for any initial velocity $\textbf{u}_0 \in \textbf{V}$ and any forcing function $\textbf{f}\in L^2(0, T; \textbf{L}^2)$. No assumption on the smallness of the data is necessary. Here $\textbf{V}$ is the space of weakly divergence free vector fields with components in $H^1$ which vanish on the boundary. We also prove the uniqueness of this strong solution. Though our modification is somewhat ad-hoc, it suggests that though more complicated, equations incorporating aspects of special and general relativity might have better existence and uniqueness properties than the ordinary Navier-Stokes equations.  
\end{abstract}
\maketitle

\section{Introduction}\label{SEC:Intro}
The ordinary incompressible Navier-Stokes equations for a fluid of constant density take the form
\begin{align*}
    \partial_t \textbf{u} -\alpha \Delta \textbf{u} + (\textbf{u}\cdot \nabla)\textbf{u} + \nabla p &= \textbf{f} \\
    \nabla \cdot \textbf{u} &=0 \\
    \textbf{u}(x, 0)&=\textbf{u}_0(x)
\end{align*}
where $\textbf{u}=(u_1, u_2, u_3)$ is the velocity, $p$ is a scalar function representing the pressure and $\textbf{f}$ is some forcing function, with all quantities depending on $(x, t)=(x_1, x_2, x_3, t)$. The incompressibility condition $\nabla \cdot \textbf{u}=0$ rules out the possibility shock or pressure waves, and experimentally the model usually holds up well up to fluid speeds of around Mach $0.3\approx 100 \,m/s$, which is vanishingly small compared to the speed of light $c\approx 3 \times 10^8 \, m/s$.

As is well known, the existence and uniqueness of so-called strong solutions the the Navier-Stokes equations in $3$-dimensions has not been established without assuming that the initial data and forcing function are "small" in some appropriate sense. We introduce a modification to these equations which allows us to prove the existence and uniqueness of strong solutions without any such smallness assumptions.

The study of relativistic fluid mechanics dates all the way back to the work of Landau and Lifshitz \cite{LandauLifshitz} in the late 1950's. Lately, the focus has been on developing relativistic hydrodynamic equations which are stable and causal \cite{Disconzi}, \cite{Hoult}, with focus on well-posedness of the Cauchy problem for such equations. Of course, considering how little we know about the Navier-Stokes equations themselves, the applications of such relativistic fluid dynamics to terrestial applications will remain limited, at least in the short term.

In this paper, we study a model which is somewhere in between the ordinary Navier-Stokes equations and fully relativistic equations (much closer to the former than the latter). The simplest and probably most well known consequence of special and general relativity is that no ordinary physical objects can exceed the speed of light. We would like to somehow embody this simple notion in a modified version of the Navier-Stokes equations. 

In a recent paper \cite{JaraczLee} we studied a modified model of visco-elastic fluids. Based on physical and experimental observations, one of the modifications to the ordinary model was the change  
\begin{equation*}
    (\textbf{u}\cdot \nabla)\textbf{u} \rightarrow (J\textbf{u}\cdot \nabla)\textbf{u}
\end{equation*}
to the convective term in the Navier-Stokes equations. Here $J$ is a regularizing operator (simply a convolution with a smooth bump function supported in some small ball which serves to average out the velocity). This allowed for the proof of the existence of global and regular solutions to the modified system. The smoothening was justified both experimentally and theoretically. Experimentally, visco-elastic flows are observed to be quite smooth. Also, theoretically the smoothing can be justified by taking Brownian motion of the fluid molecules into account where the exchange of molecules between fluid parcels serves to average out some of the velocity.  

Thus, it is natural to apply other well known physical principles to the model. We propose the following modified Navier-Stokes equations
\begin{align} \label{QRNSM} \begin{split}
    \partial_t \textbf{u} -\alpha \Delta \textbf{u} + \left(\frac{c\textbf{u}}{ \sqrt{c^2 + |\textbf{u}|^2}    }\cdot \nabla\right)\textbf{u} + \nabla p &= \textbf{f} \\
    \nabla \cdot \textbf{u} &=0 \\
    \textbf{u}(x,0)&=\textbf{u}_0(x)\end{split}
\end{align}
where $c$ is the speed of light. Notice that for $|\textbf{u}|\ll c$ we have
\begin{equation}
    \textbf{u} \approx \frac{c\textbf{u}}{ \sqrt{c^2 + |\textbf{u}|^2}    }
\end{equation}
and so for velocities much smaller than the speed of light, the model should give results very close to the ordinary Navier-Stokes equations. 

Also notice that 
\begin{equation*}
    \left\vert \frac{c\textbf{u}}{ \sqrt{c^2 + |\textbf{u}|^2}    } \right\vert \leq  c
\end{equation*}
which partially encapsulates the relativistic notion that matter can't move faster than the speed of light. Thus we call this model the quasi-relativistic Navier-Stokes equations. It is only quasi-relativistic since in principle nothing prevents $|\textbf{u}|\geq c$. Of course, this is merely an approximation to real world fluid flow, just as the Navier-Stokes themselves are. The notion of incompressibility is on its own a physical impossibility.  

Whether the model is useful will depend on how closely it models real world situations, which is to be determined. However we mention that computational rheologists often drop the convection term when numerically simulating the Navier-Stokes equations by applying certain simplifying assumptions. Thus in the cases where the convective term can be dropped, there is no difference between our model and the ordinary Navier-Stokes equations.

We will consider our equations on a bounded domain $U \subset \mathbb{R}^3$ or on the $3$-dimensional torus $\mathbb{T}$. We collectively refer to these spaces as $\Omega$, since the proof is the same in both cases. Of course we require some kind of boundary condition for the equations and so we take the Dirichlet bundary condition
\begin{equation*}
    \textbf{u}\vert_{\partial \Omega}=0
\end{equation*}
which is of course vacuously true for the empty boundary when $\Omega = \mathbb{T}$.

\subsection{Main Result} 
The notation we use for the Sobolev spaces can be found in the next section. We now state our main theorem. 

\begin{theorem} \label{TheoremOmega}
Consider the system of equations \eqref{QRNSM} with Dirichlet boundary condition on either a smooth bounded domain $U \subset \mathbb{R}^3$ or on the $3$-dimensional torus $\mathbb{T}$. Then for any $\textbf{u}_0 \in \textbf{V}$ and $ \textbf{f} \in L^2(0, T; \textbf{L}^2)$ there exists a strong solution $\textbf{u} \in L^2(0, T; \textbf{H}^2) \cap L^{\infty}(0, T; \textbf{V})$ with  $\textbf{u}' \in L^2(0, T; \textbf{L}^2)$. Moreover, this solution is unique on $0\leq t \leq T$.
\end{theorem}

There is no difference in the proofs for $\Omega$ and $\mathbb{T}$. The theorem collects the results of Propositions \ref{PropositionWeakExistence}, \ref{HigherRegularity}, and \ref{PropositionUniqueness} and thus the proof is contained therein. 

We also briefly remark on the terminology which might be slightly confusing to those who have not seen it before. A velocity satisfying 
\begin{equation*}
    \textbf{u}\in L^2(0, T; \textbf{V}), \quad \text{with} \quad \textbf{u}'\in L^2(0, T; \textbf{V}')
\end{equation*}
is called a \textit{weak} solution. On the other hand, a velocity satisfying
\begin{equation*}
  \textbf{u} \in L^2(0, T; \textbf{H}^2) \cap L^{\infty}(0, T; \textbf{V}), \quad \text{with} \quad \textbf{u}' \in L^2(0, T; \textbf{L}^2)
\end{equation*}
is referred to as a \textit{strong} solution, which is the standard nomenclature in the field. Both of these are to be contrasted with \textit{classical} solutions, which are solutions having two ordinary derivatives in space and one in time.  

\subsection{Notation}

We will be dealing with functions of both space and time, where space is $3$-dimensional. Sometimes when convenient, we will suppress one (usually the space) or both variables. Thus for example we would write $f(x, t)=f(t)=f$. Derivatives with respect to time will be denoted by $\partial_t, \partial_t^2$, etc., while derivatives with respect to space will be denoted by $\partial_{x^i}=\partial_i$ or $D^\lambda$ where $\lambda$ is a multi-index. As usual,  we use $\nabla=(\partial_{1}, \partial_{2}, \partial_{3})$ for the gradient operator, and $\Delta$ for the Laplacian.

We will let $\Omega=U, \mathbb{T}$ be the space under consideration. As mentioned, $U \subset \mathbb{R}^3$ is a smooth bounded domain and $\mathbb{T}$ is the $3$-dimensional torus. More generally, one would consider functions in $\mathbb{R}^3$ with some periodicity. Without loss of generality, we can take these functions to be $(1, 1, 1)$ periodic and then identify the cube $[0,1]^3$ with the three dimensional torus $\mathbb{T}$. Thus sometimes one can consider the system of equations as a system on $\mathbb{T}$ or as a periodic system in $\mathbb{R}^3$. Sometimes this latter point of view can be useful, though we do not use it in this paper.   

As usual, for $m\geq 1$ and $p\geq 1$ the symbol $W^{m, p}(\Omega)$ denotes the Sobolev space consisting of (real valued) functions on $\Omega$ having weak derivatives of order $m$ lying in $L^p(\Omega)$. In the case of $\mathbb{T}$ these can be lifted to periodic functions of $\mathbb{R}^3$, if desired, in the obvious way. We shall work with the case for $p=2$, and as such will denote the resulting (real) Hilbert space by $H^m(\Omega) := W^{m, 2}(\Omega)$. We also let $H_0^m(\Omega)$ denote the subspace of $H^m(\Omega)$ which consists of functions vanishing on $\partial \Omega$ in the trace sense. Of course, for the torus $H_0^m(\mathbb{T})= H^m(\mathbb{T})$. Since the proof of our theorem is the same for both $U$ and $\mathbb{T}$ we suppress $\Omega$ from the notation, and simply write $L^2$, $H^m$, etc.

We will be dealing with velocities having three components so we define
\begin{equation*}
\textbf{H}^m = H^m \times H^m \times H^m 
\end{equation*} 
with the Hilbert norm given by 
\begin{equation*}
\Vert \textbf{f}\Vert_{\textbf{H}^m}^2 = \|\left(f_1, f_2, f_3 \right) \|_{\mathbb{H}^m}^2 = \|f_1\|_{H^m}^2 + \|f_2\|_{H^m}^2 + \|f_3\|_{H^m}^2 . 
\end{equation*}
We also have
\begin{equation*}
\textbf{H}^m_0 = H^m_0 \times H^m_0 \times H^m_0 
\end{equation*} 
with the same norm. For $m=0$ we write $\textbf{L}^2$.

We are interested in the case of incompressible flows for which the velocity $\textbf{u}$ satisfies
\begin{equation*}
\nabla \cdot \textbf{u} = 0.
\end{equation*}
Since our velocities lie in $\textbf{H}^m$, where we mean that the above condition holds in the weak sense. Therefore, we define
\begin{equation*}
\mathcal{V} = \lbrace \textbf{v} \in C_c^\infty (\Omega) \times C_c^\infty (\Omega) \times C_c^\infty (\Omega) \; | \; \nabla \cdot \textbf{v}=0 \rbrace \subset \textbf{H}^1_0, 
\end{equation*}
and then take the closure of this subspace with respect to the $\textbf{H}^1_0$ norm. We denote the resulting space by
\begin{equation*}
\textbf{V} \coloneqq \overbar{\mathcal{V}} \subset \textbf{H}^1_0, 
\end{equation*}
and so 
\begin{equation*}
\textbf{V} = \lbrace \textbf{u} \in \textbf{H}^1_0 \; | \; \nabla \cdot \textbf{u} =0 \; \text{almost everywhere} \rbrace.
\end{equation*}
Also, we let $\textbf{V}' $ be the dual space of $\textbf{V} $. 
We also define 
\begin{equation*}
    H \coloneqq \overbar{\mathcal{V}} \subset \textbf{L}^2
\end{equation*}
meaning the closure of $\mathcal{V}$ with respect to the $\textbf{L}^2$ norm. This notation is used to be similar to the notation used for these spaces in \cite{temam}.

Also we let
\begin{equation*}
    |\nabla \textbf{u}|^2 \coloneqq \sum_{i, j=1}^3 |\partial_j u_i|^2
\end{equation*}
and then denote
\begin{equation*}
    \Vert \nabla \textbf{u} \Vert_{\textbf{L}^2}^2 = \int_\Omega  |\nabla \textbf{u}|^2 \, dx
\end{equation*}
and so we can write
\begin{equation*}
    \Vert \textbf{u} \Vert_{\textbf{H}^1}^2 = \Vert \textbf{u} \Vert_{\textbf{L}^2}^2 + \Vert \nabla \textbf{u} \Vert_{\textbf{L}^2}^2
\end{equation*}

Oftentimes, when there is no chance of confusion, we will simply write 
\begin{equation*}
    \Vert \cdot \Vert_{\textbf{L}^2} = \Vert \cdot \Vert
\end{equation*}
so unless a space is specified, the norm is assumed to be the $\textbf{L}^2$ norm. We also write
\begin{equation*}
    (\cdot, \cdot)_{\textbf{L}^2}=(\cdot, \cdot)
\end{equation*}
for the $\textbf{L}^2$ inner product defined in the obvious way. 
Finally, $\langle \textbf{u}, \textbf{w} \rangle$ denotes the pairing for $\textbf{u}\in \textbf{V}$ and $\textbf{w} \in \textbf{V}'$.

We shall also use the standard notation for spaces involving time. Let $Z$ denote a real Banach space. The space $L^p(0, T; Z)$ consists of all measurable functions $\textbf{u}: \left[ 0, T \right] \rightarrow Z$ with
\begin{equation*}
\Vert \textbf{u} \Vert_{L^p(0, T; Z)} \coloneqq \left( \int_{0}^{T} \Vert \textbf{u}(t) \Vert^p_Z \, dt \right)^{1/p} < \infty
\end{equation*} 
for $1\leq p<\infty$, and 
\begin{equation*}
    \Vert \textbf{u} \Vert_{L^\infty (0, T; Z)} \coloneqq \esssup_{0\leq t\leq T} \Vert \textbf{u}(t) \Vert_Z < \infty
\end{equation*}
for $p=\infty$. The space $C(\left[ 0, T \right]; Z)$
comprises all continuous functions $\textbf{u}(t):[0, T]\rightarrow Z$ with
\begin{equation*}
\Vert \textbf{u} \Vert_{C([0, T]; Z)} \coloneqq \max_{0\leq t \leq T} \Vert \textbf{u}(t) \Vert_Z < \infty. 
\end{equation*}
Now let $\textbf{u} \in L^1(0, T; Z)$. We say $\textbf{v} \in L^1(0, T; Z)$ is the weak time derivative of $\textbf{u}$ provided that
\begin{equation*}
\int_{0}^{T} w'(t)\textbf{u}(t) \, dt=-\int_{0}^{T} w(t)\textbf{v}(t) \, dt, 
\end{equation*}
for all scalar test functions $w \in C^\infty_c(0, T)$. The Sobolev space $
W^{1, p}\left( 0, T; Z \right)$ consists of all functions $\textbf{u}\in L^p(0, T; Z)$ such that $\textbf{u}'$ exists in the weak sense and belongs to $L^p(0, T; Z)$. Furthermore,
\begin{equation*}
\Vert \textbf{u} \Vert_{W^{1, p}(0, T; Z)} \coloneqq 
\begin{cases} 
\left( \int_{0}^{T} \Vert \textbf{u}(t) \Vert_Z^p + \Vert \textbf{u}'(t)\Vert_Z^p \, dt \right)^{1/p} & \left( 1\leq p< \infty\right) \\
\esssup_{0\leq t\leq T} \left( \Vert \textbf{u}(t) \Vert_Z + \Vert \textbf{u}'(t) \Vert_Z \right)& \left(p=\infty \right). 
\end{cases}
\end{equation*}
In addition, we write $H^1(0, T; Z)=W^{1, 2}(0, T; Z)$. 

As usual, since our spaces are products of separable Hilbert spaces, they are Hilbert spaces themselves. Thus each of these spaces has some countable orthonormal basis, and thus solutions to PDEs can be constructed in the usual way using Galerkin approximations.

\section{Proof of Theorem \ref{TheoremOmega}   }
As mentioned, the proof is the same for $U$ or $\mathbb{T}$ and so we collectively refer to these spaces as $\Omega$.
 The method of proof follows Section 7.1 in \cite{evans} and Chapter 3 in \cite{temam}. As such, for those familiar, the proof can be viewed as merely a review of standard PDE methods. Nevertheless, since the model is new and we have introduced quite a strong non-linearity, we do need to check that all of the standard methods work, and thus we have spared no detail.

\subsection{Existence of a Weak Solution}
For the moment we assume that 
\begin{equation} \label{fCondition}
    \textbf{f}\in L^2(0, T; \textbf{V}')
\end{equation}
and this is the minimum regularity on $\textbf{f}$ we require. Later on, when we wish to obtain higher regularity for our solution, we will assume a stronger condition on $\textbf{f}$. 

We now wish to weakly solve the system of equations
\begin{align} \label{QRNSMOmega} \begin{split}
    \partial_t \textbf{u} -\alpha \Delta \textbf{u} + \left(\frac{c\textbf{u}}{ \sqrt{c^2 + |\textbf{u}|^2}    }\cdot \nabla\right)\textbf{u} + \nabla p &= \textbf{f} \\
    \nabla \cdot \textbf{u} &=0 \\
    \textbf{u}(x,0)&=\textbf{u}_0(x) \\
    \textbf{u} \vert_{\partial \Omega} &= 0 
    \end{split}
\end{align}
in $\Omega$ with Dirichlet boundary condition. There are two equivalent ways of defining what we mean by this. 
First we define
\begin{equation*}
    A[\textbf{u}, \textbf{v}] \coloneqq \alpha \sum_{i=1}^3 \int_{\Omega} \nabla u_i \cdot \nabla v_i \, dx
\end{equation*}
and 
\begin{equation*}
    B[\textbf{y}, \textbf{u}, \textbf{v}]\coloneqq \int_{\Omega} \left[ \left(\frac{c\textbf{y}}{ \sqrt{c^2 + |\textbf{y}|^2}    }\cdot \nabla\right)\textbf{u} \right] \cdot \textbf{v} \, dx = \sum_{i=1}^3 \sum_{j=1}^3 \int_{\Omega} \frac{c y_j (\partial_j u_i) v_i  }   { \sqrt{c^2 + |\textbf{y}|^2}    } \, dx.
\end{equation*}
Notice that $A[\textbf{u}, \textbf{v}]$ is linear in both arguments, while $ B[\textbf{y}, \textbf{u}, \textbf{v}]$ is linear in $\textbf{u}$ and $\textbf{v}$ but not in $\textbf{y}$. However, we see we can make the estimate
\begin{equation} \label{BBound}
    \left|B[\textbf{y}, \textbf{u}, \textbf{v}]\right| \leq C_B \int_{\Omega} |\nabla \textbf{u}| |\textbf{v}| \, dx \leq C_B \Vert \nabla \textbf{u} \Vert \Vert \textbf{v} \Vert
\end{equation}
where the constant $C_B>0$ depends only on $c$ and the dimension, which is $3$ in our case. In addition we have the estimate
\begin{equation} \label{ABound}
    |A[\textbf{u}, \textbf{v}]| \leq \alpha \Vert \nabla \textbf{u} \Vert \Vert \nabla \textbf{v} \Vert
\end{equation}

Let $\langle \cdot, \cdot \rangle$ denote the bilinear pairing between $\textbf{V}'$ and $\textbf{V}$. For $\textbf{u}\in \textbf{V}$ we see that \begin{equation*}
    A[\textbf{u}, \cdot], \, B[\textbf{u}, \textbf{u}, \cdot] \in \textbf{V}'
\end{equation*}
and so we can express these as elements $A(\textbf{u})$ and $B(\textbf{u})$ such that
\begin{equation*}
    \langle A(\textbf{u}), \textbf{v}\rangle \coloneqq A[\textbf{u}, \textbf{v}], \quad  \langle B(\textbf{u}), \textbf{v}\rangle \coloneqq B[\textbf{u}, \textbf{v}]
\end{equation*}

One way of defining a weak solution is to find a function 
\begin{equation*}
     \textbf{u}\in L^2(0, T; \textbf{V})
\end{equation*}
which satisfies
\begin{align} \label{weakForm} 
     \frac{d}{dt}(\textbf{u}, \textbf{v} ) + A[\textbf{u}, \textbf{v}] + B[\textbf{u}, \textbf{u}, \textbf{v}]&=\langle \textbf{f}, \textbf{v} \rangle \\ \label{weakFormU}
\textbf{u}(0)&=\textbf{u}_0 
\end{align}
for $\textbf{f}  \in \textbf{V}' $ for any $\textbf{v} \in \textbf{V}$ for almost every $0\leq t \leq T$.
Notice that formally due to the incompressibility of $\textbf{v}$ the term containing the pressure disappears. 

Another way of defining weak solutions is to find functions
\begin{equation*}
    \textbf{u}\in L^2(0, T; \textbf{V}), \quad \text{with} \quad \textbf{u}'\in L^2(0, T; \textbf{V}')
\end{equation*}
such that 
\begin{align} \label{weakForm2} 
    \textbf{u}'+A(\textbf{u})+B(\textbf{u})&=\textbf{f} \\ \label{weakFormU2}
\textbf{u}(0)&=\textbf{u}_0 
\end{align}
for $\textbf{f}  \in \textbf{V}' $ for a.e. time $0\leq t \leq T$. In view of \eqref{BBound} and \eqref{ABound} it is easy to show that these two formulations are actually equivalent.

\begin{proposition} \label{PropositionWeakExistence}
Suppose that $\textbf{f}\in L^2(0, T; \textbf{V}\,')$ and $\textbf{u}_0 \in H$. Then the problem \eqref{weakForm}, \eqref{weakFormU} has a weak solution with $ \textbf{u} \in L^2(0, T; \textbf{V}) \cap L^\infty (0, T; H), \,   \textbf{u}' \in L^2(0, T; \textbf{V}\,')$. Moreover, $\textbf{u}\in C([0, T]; H)$.
\end{proposition}

\begin{proof}

We construct approximate solutions using the Galerkin method, derive the appropriate energy estimates, and take the limit. Since $\textbf{V}$ is a separable Hilbert space and $\mathcal{V}$ is dense in $\textbf{V}$, we can select a countable basis 
\begin{equation*}
    \lbrace \textbf{w}_i \rbrace_{i=1}^{\infty} \subset \mathcal{V} \subset \textbf{V}
\end{equation*}
and moreover we can choose this basis so that it is orthonormal in  $\textbf{V}$. We now construct approximate solutions of the form
\begin{equation*}
    \textbf{u}_m = \sum_{i=1}^m g_{im}(t) \textbf{w}_i
\end{equation*}
where the functions $g_{im}(t)$ satisfy the a certain system of differential equations. Before that however, let $\textbf{u}_{0m}$ be the projection of $\textbf{u}_0$ onto $\text{span} \lbrace \textbf{w}_1, \dots, \textbf{w}_m \rbrace$. Now we write  
\begin{equation*}
    \textbf{u}_{m0}=\sum_{i=1}^m c_i \textbf{w}_i
\end{equation*}
To obtain the system of equations for the $g_{im}$ we substitute $\textbf{u}_m$ for $\textbf{u}$ and $\textbf{w}_k$ for $\textbf{v}$ in \eqref{weakForm} which yields
\begin{align} \label{WeakForm2}
(\textbf{u}_m', \textbf{w}_k) + A[\textbf{u}_m, \textbf{w}_k] + B[\textbf{u}_m, \textbf{u}_m, \textbf{w}_k]&=\langle \textbf{f}, \textbf{w}_k \rangle \\ \label{WeakForm2u}
\textbf{u}_m(0)&=\textbf{u}_{m0} 
\end{align}
which therefore gives the system of equations
\begin{align} \label{ODE} \begin{split}
    \sum_{i=1}^m (\textbf{w}_i, \textbf{w}_k) g'_{im}(t)  + \sum_{i=1}^m g_{im} A[\textbf{w}_i, \textbf{w}_k] + \sum_{i=1}^m g_{im} B\left[\textbf{u}_m, \textbf{w}_i, \textbf{w}_k\right] &= f_k(t) \\
    g_{km}(0)&=c_k \end{split}
    \end{align}
where $f_k(t)\coloneqq \langle \textbf{f}(t), \textbf{w}_k \rangle$. It is easy to see that 
\begin{equation*}
    \sum_{i=1}^m g_{im} B\left[\textbf{u}_m, \textbf{w}_i, \textbf{w}_k\right] = \sum_{j=i}^m g_{im} B\left[\sum_{j=1}^m g_{jm}(t) \textbf{w}_j, \textbf{w}_i, \textbf{w}_k\right]
\end{equation*}
is smooth in the $g_{lm}$ (this is especially easy to see since $\textbf{w}_i$ were chosen to be smooth). 

Moreover, since the $\textbf{w}_i$ are linearly independent, the matrix $\beta$ with components
\begin{equation*}
    \beta_{ij} \coloneqq (\textbf{w}_i, \textbf{w}_j)
\end{equation*}
is nonsingular and thus has an inverse $\beta^{-1}=\beta^{ij}$. Multiplying \eqref{ODE} by $\beta^{-1}$ yields a first order system of ODEs which is smooth in the $g_{im}$. As a result, the system of equations \eqref{ODE} has a unique absolutely continuous solution existing on some maximal interval $[0, t_m]$. If $t_m<T$, then it must be that $\Vert \textbf{u}_m(t) \Vert \rightarrow \infty$ as $t \nearrow t_m$. We will now derive the usual apriori estimates which will show this does not happen, and so in fact $t_m=T$ for all $m$.  

We take the first equation of \eqref{WeakForm2}, multiply by $g_{km}$ and sum over $k$ to obtain
\begin{equation*}
      \frac{1}{2} \frac{d}{dt} \Vert \textbf{u}_m (t) \Vert^2 + A[\textbf{u}_m, \textbf{u}_m] + B[\textbf{u}_m, \textbf{u}_m, \textbf{u}_m]=\langle \textbf{f}, \textbf{u}_m \rangle 
\end{equation*}
which we rewrite as 
\begin{equation*}
      \frac{1}{2} \frac{d}{dt} \Vert \textbf{u}_m (t) \Vert^2 + \alpha \Vert \nabla \textbf{u}_m \Vert^2 =\langle \textbf{f}, \textbf{u}_m \rangle  -  B[\textbf{u}_m, \textbf{u}_m, \textbf{u}_m].
\end{equation*}
Using \eqref{BBound} we can then write
\begin{align*}
      \frac{1}{2} \frac{d}{dt} \Vert \textbf{u}_m (t) \Vert^2 &+ \alpha \Vert \nabla \textbf{u}_m \Vert^2 \leq \Vert \textbf{f} \Vert_{\textbf{V}'} \Vert \textbf{u}_m \Vert_{ \textbf{V}} + C_B \Vert \nabla \textbf{u}_m \Vert \Vert \textbf{u}_m \Vert \\
       &\leq \frac{\Vert \textbf{f} \Vert^2_{\textbf{V}'}}{4\varepsilon} +  \varepsilon\Vert \textbf{u}_m \Vert_{ \textbf{V}}^2 + C_B \varepsilon \Vert \nabla \textbf{u}_m \Vert^2 +\frac{C_B \Vert \textbf{u}_m \Vert^2}{4\varepsilon} \\
       &= \frac{\Vert \textbf{f} \Vert^2_{\textbf{V}'}}{4\varepsilon} +  \varepsilon(\Vert \textbf{u}_m \Vert^2+ \Vert \nabla \textbf{u}_m \Vert^2) + C_B \varepsilon \Vert \nabla \textbf{u}_m \Vert^2 +\frac{C_B \Vert \textbf{u}_m \Vert^2}{4\varepsilon} \\
       &=\frac{\Vert \textbf{f} \Vert^2_{\textbf{V}'}}{4\varepsilon} +    (1+C_B) \varepsilon \Vert \nabla \textbf{u}_m \Vert^2 +\left(\varepsilon+\frac{C_B}{4\varepsilon}\right) \Vert \textbf{u}_m \Vert^2
\end{align*}
where we used Cauchy's inequality with $\varepsilon$. Choosing $\varepsilon$ so that that $(1+C_B) \varepsilon=\alpha/2$, rearranging and multiplying by $2$ we obtain
\begin{align} \label{Inequality1}
    \frac{d}{dt} \Vert \textbf{u}_m (t) \Vert^2 + \alpha \Vert \nabla \textbf{u}_m \Vert^2 \leq K \left(\Vert \textbf{f} \Vert_{\textbf{V}'}^2 +  \Vert \textbf{u}_m (t) \Vert^2 \right)
\end{align}
where the constant $K>0$ depends only on $C_B$ and $\alpha$.

In particular, we have 
\begin{align*}
     \frac{d}{dt} \Vert \textbf{u}_m (t) \Vert^2  \leq K \left(\Vert \textbf{f} \Vert_{\textbf{V}'}^2 +  \Vert \textbf{u}_m (t) \Vert^2 \right)
\end{align*}
and thus applying Gronwall's inequality we obtain
\begin{align*}
     \Vert \textbf{u}_m (t) \Vert^2 &\leq e^{Kt} \left( \Vert \textbf{u}_m (0) \Vert^2 + K \int_0^t \Vert \textbf{f} (s) \Vert_{\textbf{V}'}^2 \, ds  \right)  
     \\ & \leq e^{Kt} \left( K+1 \right)  \left( \Vert \textbf{u}_m (0) \Vert^2 + \int_0^t \Vert \textbf{f} (s) \Vert_{\textbf{V}'}^2 \, ds  \right)  .   
\end{align*}
Since
\begin{equation*}
    \Vert \textbf{u}_m (0) \Vert^2 \leq \Vert \textbf{u}_0 \Vert^2
\end{equation*}
we have
\begin{equation} \label{maxum}
    \sup_{0\leq t \leq T}  \Vert \textbf{u}_m (t) \Vert^2  \leq e^{KT} \left( K+1 \right)  \left( \Vert \textbf{u}_0 \Vert^2 + \int_0^t \Vert \textbf{f} (s) \Vert_{\textbf{V}'}^2 \, ds  \right) \coloneqq C_1(T)
\end{equation}
where in particular $C_1(T)$ is independent of $m$. Next, returning to \eqref{Inequality1} and integrating with respect to $t$ we obtain
\begin{Small}\begin{align} \label{C2} \begin{split}
    \int_0^T \Vert \nabla \textbf{u}_m(t) \Vert^2 \, dt &\leq \frac{1}{\alpha} \left(   \Vert \textbf{u}_m (0) \Vert^2-  \Vert \textbf{u}_m (T) \Vert^2 +K \int_0^T \Vert \textbf{f}(t) \Vert_{\textbf{V}'}^2 \, dt  +  K \int_0^T  \Vert \textbf{u}_m (t) \Vert^2 \, dt  \right) \\
    &\leq  \frac{1}{\alpha} \left( 2C_1(T) +  KT C_1(T)  + K\int_{0}^T \Vert \textbf{f}(t) \Vert_{\textbf{V}'}^2 \, dt  \ \right) \coloneqq C_2(T) \end{split}
\end{align}
\end{Small}
and so we can estimate
\begin{equation} \label{Boundum}
    \Vert \textbf{u}_m \Vert^2_{L(0, T; \textbf{V})} \leq TC_1(T)+ C_2(T)
\end{equation}
where the right hand side is independent of $m$.

As a result of \eqref{maxum} and \eqref{Boundum} there exists a function
\begin{equation*}
    \textbf{u} \in L^2(0, T; V) \cap L^\infty (0, T; H) 
\end{equation*}
and a subsequence $m'$ such that
\begin{align} \label{weakConvergence}
    \left. \begin{array}{l}
    \textbf{u}_{m'} \rightarrow \textbf{u} \quad \text{in} \quad L^2(0, T; V) \quad \text{weakly, and}\\ 
    L^\infty (0, T; H) \quad \text{weak-star, as} \quad m'\rightarrow \infty
    \end{array}    \right\rbrace.
\end{align}
Moreover, it can be shown that in fact
\begin{equation} \label{StrongConvergenceInH}
    \textbf{u}_{m'} \rightarrow \textbf{u} \quad \text{in} \quad L^2(0, T; H) \quad \text{strongly}.
\end{equation}
The proof is the same as the proof of (3.3.41) in \cite{temam}.

Next, we need to show that $\textbf{u}$ in fact is a solution of \eqref{weakForm}. To do so, we let $\psi$ be a continuously differntiable function on $[0, T]$ with $\psi(T)=0$. We multiply \eqref{WeakForm2} by $\psi$ and integrate by parts to obtain 
\begin{align*}
    -\int_0^T ( \textbf{u}_m (t), \psi'(t) \textbf{w}_k) \, dt &+ 
\int_0^T A[\textbf{u}_m(t), \psi(t) \textbf{w}_k] \, dt + \int_0^T B[\textbf{u}_m(t), \textbf{u}_m(t), \psi(t) \textbf{w}_k] \, dt \\ \quad &= (\textbf{u}_{m0}, \textbf{w}_k)\psi(0) + \int_0^T \langle \textbf{f}(t),  \psi(t) \textbf{w}_k \rangle \, dt.
\end{align*}
Passing to the limit for the subsequence $m'$ is trivial for the linear terms by definition of weak convergence. We  have to deal with the nonlinear term. This can be done with the use of the following lemma, whose proof is elementary.

\begin{lemma} \label{WeakBanachConvergence}
Let $X$ be a Banach space, $\lbrace x_m \rbrace \subset X$ a sequence of elements such that $x_m \rightarrow x$ weakly, and $\lbrace l_k \rbrace \subset X^*$ a sequence of linear functionals such that $l_k \rightarrow l$ in the operator norm. Then $l_k(x_k) \rightarrow l(x)$.
\end{lemma}

With this lemma in hand, define $\Omega_T=\Omega \times [0, T]$. Let $\textbf{v}=\psi(t)\textbf{w}_k$. Let
\begin{equation} \label{vConstant}
    C_v = \sup_{(x, t)\in \Omega_T} |\psi(t) \textbf{w}_k(x)|
\end{equation}
which exists since $\psi(t) \textbf{w}_k(x) \in C^1(\overbar{\Omega}_T)$ by the smoothness of the components of $\textbf{w}_k$. For such $\textbf{v}$ we have
\begin{equation*}
    \int_0^T B[\textbf{u}_{m'}, \cdot, \textbf{v}] \, dt \in ( L^2(0, T; V))^*
\end{equation*}
where using \eqref{BBound} and \eqref{vConstant} we have
\begin{equation*}
    \left\vert \int_0^T B[\textbf{u}_{m'}, \cdot, \textbf{v}] \, dt \right\vert \leq C_B C_v \sqrt{|\Omega| T} \Vert \cdot \Vert_{L^2(0, T; V)}
\end{equation*}
where $|\Omega|$ is the volume of $\Omega$.

Now, we want to show that for $\textbf{v}=\psi(t) \textbf{w}_k$ 
\begin{equation} \label{OperatorConvergence}
    \int_0^T B[\textbf{u}_{m'}, \cdot, \textbf{v}] \, dt \rightarrow \int_0^T B[\textbf{u}, \cdot, \textbf{v}] \, dt
\end{equation}
in the operator norm. 

The easiest way to do this is as follows. Suppose that
\begin{equation*}
    \textbf{y} \in L^2(0, T; H)
\end{equation*}
then we have
\begin{equation*}
    \Vert \textbf{y} \Vert_{L^2(0, T; H)} = \left( \int_0^T \Vert \textbf{y} \Vert_{\textbf{L}^2} \, dt    \right)^{1/2} = \left( \int_0^T \int_\Omega |\textbf{y}|^2 \, dx \, dt    \right)^{1/2} = \Vert \textbf{y} \Vert_{\textbf{L}^2(\Omega_T)}
\end{equation*}
and thus by \eqref{StrongConvergenceInH} we have
\begin{equation*}
    \textbf{u}_{m'} \rightarrow \textbf{u} \quad \text{strongly in} \quad \textbf{L}^2(\Omega_T). 
\end{equation*}
It is then well known that (passing to a subsequence if necessary which we continue to denote by $m'$)
\begin{equation} \label{ConvergenceAE}
    \textbf{u}_{m'}(x, t) \rightarrow \textbf{u}(x, t) \quad \text{almost everywhere in} \; \Omega_T. 
\end{equation}
Now, take an arbitrary $\textbf{y} \in L^2(0, T; V)$ and consider
\begin{Small}
\begin{align*}
     &\left\vert \int_0^T B[\textbf{u}_{m'}, \textbf{y}, \textbf{v}] \, dt - \int_0^T B[\textbf{u},  \textbf{y}, \textbf{v}] \, dt \right\vert \leq \sum_{i=1}^3 \sum_{j=1}^3 \left\vert \int_0^T  \int_{\Omega} \left( \frac{c (u_{m'})_j (\partial_j y_i) v_i  }   { \sqrt{c^2 + |\textbf{u}_{m'}|^2}    } - \frac{c u_j (\partial_j y_i) v_i  }   { \sqrt{c^2 + |\textbf{u}|^2}    } \right) \, dx \, dt \right\vert \\
     &\hspace{4cm}\leq \sum_{i=1}^3 \sum_{j=1}^3 C_v  \int_0^T  \int_{\Omega} \left\vert \frac{c (u_{m'})_j   }   { \sqrt{c^2 + |\textbf{u}_{m'}|^2}    } - \frac{c u_j   }   { \sqrt{c^2 + |\textbf{u}|^2}    }  \right\vert |\partial_j y_i| \, dx \, dt \\ &\hspace{4cm} \leq 3C_v \left( \sum_{j=1}^3 \int_T \int_\Omega  \left\vert \frac{c (u_{m'})_j   }   { \sqrt{c^2 + |\textbf{u}_{m'}|^2}    } - \frac{c u_j   }   { \sqrt{c^2 + |\textbf{u}|^2}    }  \right\vert^2  \, dx \, dt \right)^{1/2} \Vert \textbf{y} \Vert_{L^2(0, T; V)}
\end{align*}
\end{Small}
as can be easily checked by applying H\"{o}lder's inequality. Now we can estimate
\begin{equation*}
    \left\vert \frac{c (u_{m'})_j   }   { \sqrt{c^2 + |\textbf{u}_{m'}|^2}    } - \frac{c u_j   }   { \sqrt{c^2 + |\textbf{u}|^2}    }  \right\vert^2 \leq 4c^2
\end{equation*}
and so by \eqref{ConvergenceAE} and the dominated convergence theorem we obtain
\begin{equation*}
    \lim_{m' \rightarrow \infty} \int_T \int_\Omega  \left\vert \frac{c (u_{m'})_j   }   { \sqrt{c^2 + |\textbf{u}_{m'}|^2}    } - \frac{c u_j   }   { \sqrt{c^2 + |\textbf{u}|^2}    }  \right\vert^2  \, dx \, dt = 0
\end{equation*}
and so \eqref{OperatorConvergence} does indeed hold. As a result, by Lemma \ref{WeakBanachConvergence} we obtain
\begin{equation} \label{TermConvergence}
    \lim_{m' \rightarrow \infty }\int_0^T B[\textbf{u}_{m'}, \textbf{u}_{m'}, \textbf{v}] \, dt \rightarrow \int_0^T B[\textbf{u}, \textbf{u}, \textbf{v}] \, dt
\end{equation}
for any $\textbf{v}$ of the form $\textbf{v}=\psi(t) \textbf{w}_k$. As a result, \eqref{TermConvergence} holds for any linear combination of such functions, and by a continuity argument it holds for any $\textbf{v}\in L^2(0, T; V)$. Hence in the limit we obtain
\begin{align} \label{EquationInLimit} \begin{split}
    -\int_0^T ( \textbf{u} (t), \psi'(t) \textbf{v}) \, dt &+ 
\int_0^T A[\textbf{u}(t), \psi(t) \textbf{v}] \, dt + \int_0^T B[\textbf{u}(t), \textbf{u}(t), \psi(t) \textbf{v}] \, dt \\ \quad &= (\textbf{u}_{0}, \textbf{v})\psi(0) + \int_0^T \langle \textbf{f}(t),  \psi(t) \textbf{v} \rangle \, dt \end{split}
\end{align}
for any $v\in V$. Thus \eqref{weakForm} holds in the weak sense.

The last thing which we need to verify is that $\textbf{u}(0)=\textbf{u}_0$. Taking \eqref{weakForm}, multiplying by $\psi(t)$ and integrating, we obtain
\begin{align*} \begin{split}
    -\int_0^T ( \textbf{u} (t), \psi'(t) \textbf{v}) \, dt &+ 
\int_0^T A[\textbf{u}(t), \psi(t) \textbf{v}] \, dt + \int_0^T B[\textbf{u}(t), \textbf{u}(t), \psi(t) \textbf{v}] \, dt \\ \quad &= (\textbf{u}(0), \textbf{v})\psi(0) + \int_0^T \langle \textbf{f}(t),  \psi(t) \textbf{v} \rangle \, dt \end{split}
\end{align*}
which upon comparing with \eqref{EquationInLimit} we obtain
\begin{equation*}
    (\textbf{u}(0)-\textbf{u}_0, \textbf{v})\psi(0)=0
\end{equation*}
and by choosing $\psi(t)$ with $\psi(0)=1$ we get 
\begin{equation*}
    (\textbf{u}(0)-\textbf{u}_0, \textbf{v})=0, \quad \forall \quad  \textbf{v}\in V
\end{equation*}
from which \eqref{weakFormU} follows.

Using the usual identifications we can write
\begin{equation*}
    (\textbf{u}, \textbf{v}) = \langle \textbf{u}, \textbf{v} \rangle
\end{equation*}
and thus by Lemma 3.1.1 in \cite{temam} we can write \eqref{weakForm} as 
\begin{equation*}
    \frac{d}{dt}(\textbf{u}, \textbf{v}) = \frac{d}{dt}\langle \textbf{u}, \textbf{v} \rangle=\langle \textbf{u}', \textbf{v} \rangle = \langle \textbf{f} - A(\textbf{u})-B(\textbf{u}), \textbf{v} \rangle  
\end{equation*}
so
\begin{equation*}
    \textbf{u}'=\textbf{f} - A(\textbf{u})-B(\textbf{u})
\end{equation*}
for almost every $t$. By \eqref{BBound}, \eqref{ABound}, and \eqref{Boundum} we see that
\begin{equation*}
    \textbf{u}' \in L^2(0, T; \textbf{V}')
\end{equation*}
with a bound of the form 
\begin{equation}
    \Vert \textbf{u}' \Vert_{L^2(0, T; \textbf{V}')} \leq C\left( \Vert \textbf{u} \Vert_{ L^2(0, T, \textbf{V}     } + \Vert \textbf{f} \Vert_{ L^2(0, T, \textbf{V}'     }  \right)
\end{equation}
for some constant $C$. Therefore, by Theorem 3 in Section 5.9.2 in \cite{evans} we have  $\textbf{u}\in C([0, T]; H)$.
\end{proof}

Now that we have obtained that a weak solution exists, we would like to obtain better regularity for the solution.

\subsection{Improved Regularity} We have the following proposition which corresponds to Theorem 5 in \cite{evans}.

\begin{proposition} \label{HigherRegularity}
\noindent Assume that
    \begin{equation*}
        \textbf{u}_0 \in \textbf{V}, \quad \textbf{f} \in L^2(0, T; \textbf{L}^2).
    \end{equation*}
Suppose also $\textbf{u}\in L^2(0, T; \textbf{V})$, with $\textbf{u}' \in L^2(0, T, \textbf{V}\,')$ is the weak solution of problem \eqref{QRNSMOmega}. Then in fact
\begin{equation*}
    \textbf{u} \in L^2(0, T; \textbf{H}^2) \cap L^{\infty}(0, T; V), \quad \textbf{u}' \in L^2(0, T; \textbf{L}^2) 
    \end{equation*}
    and we have the estimate
    \begin{equation} \label{EstimateSupV}
        \esssup_{0\leq t\leq T} \Vert \textbf{u}(t) \Vert_V + \Vert \textbf{u} \Vert_{L^2(0, T; \textbf{H}^2)} + \Vert \textbf{u}' \Vert_{L^2(0, T; \textbf{L}^2)} \leq C\left(\Vert \textbf{f} \, \Vert_{L^2(0, T; \textbf{L}^2)} + \Vert \textbf{u}_0 \Vert_V \right)
\end{equation}
where the constant $C$ depends only on $c, \alpha, T$ and $\Omega$.

\end{proposition}

\begin{proof}
If $\textbf{f}\in \textbf{L}^2$ then using the usual identifications we can write $\langle \textbf{f}, \textbf{v} \rangle =(\textbf{f}, \textbf{v})$. We fix $m\geq 1$, we multiply equation \eqref{WeakForm2} by $g_{km}'(t)$ and sum over $k$ to obtain 
\begin{equation*}
    \left(  \textbf{u}_m' ,  \textbf{u}_m'  \right) + A[\textbf{u}_m, \textbf{u}_m'] + B[\textbf{u}_m, \textbf{u}_m, \textbf{u}_m'] = (\textbf{f}, \textbf{u}_m' ).
\end{equation*}
We see that 
\begin{equation*}
    A[\textbf{u}_m, \textbf{u}_m'] = \frac{d}{dt} \left( \frac{1}{2} A[\textbf{u}_m, \textbf{u}_m]   \right) = \frac{d}{dt} \left( \frac{\alpha}{2} \int_{\Omega} |\nabla \textbf{u}_m |^2 \, dx      \right)
\end{equation*}
and 
\begin{equation*}
    |B[\textbf{u}_m, \textbf{u}_m, \textbf{u}_m']| \leq C_B \Vert \nabla \textbf{u}_m \Vert_{\textbf{L}^2} \Vert \textbf{u}_m' \Vert_{\textbf{L}^2} \leq C_B \Vert \textbf{u}_m \Vert_{\textbf{V}} \Vert \textbf{u} \Vert_{\textbf{L}^2}
\end{equation*}
and therefore
\begin{equation*}
    |B[\textbf{u}_m, \textbf{u}_m, \textbf{u}_m']| \leq \frac{C_B^2}{4\epsilon} \Vert \textbf{u}_m \Vert_{\textbf{V}}^2 + \epsilon \Vert \textbf{u}_m' \Vert_{\textbf{L}^2}^2, \quad |(\textbf{f}, \textbf{u}_m')| \leq \frac{1}{4\epsilon}\Vert \textbf{f} \Vert_{\textbf{L}^2}^2 + \epsilon \Vert \textbf{u}_m' \Vert_{\textbf{L}^2}^2
\end{equation*}
for any $\epsilon>0$. Thus we obtain
\begin{equation*}
    \Vert \textbf{u}_m' \Vert_{\textbf{L}^2}^2 + \frac{d}{dt} \left( \frac{1}{2} A[\textbf{u}_m, \textbf{u}_m] \right)  
     \leq \frac{C_B^2}{4\epsilon} \Vert \textbf{u}_m \Vert_{\textbf{V}}^2 + \frac{1}{4\epsilon} \Vert \textbf{f} \Vert_{\textbf{L}^2}^2 + 2\epsilon \Vert \textbf{u}_m' \Vert_{\textbf{L}^2}^2 
\end{equation*}
and upon letting $\epsilon=1/4$, subtracting and multiplying by $2$ we obtain
\begin{equation*}
     \Vert \textbf{u}_m' \Vert_{\textbf{L}^2}^2 + \frac{d}{dt} \left( A[\textbf{u}_m, \textbf{u}_m] \right)  
     \leq 2C_B^2 \Vert \textbf{u}_m \Vert_{\textbf{V}}^2 + 2 \Vert \textbf{f} \Vert_{\textbf{L}^2}^2. 
\end{equation*}
Therefore, integrating we obtain
\begin{Small}
\begin{align*}
    \int_0^T \Vert \textbf{u}_m'(t) \Vert_{\textbf{L}^2}^2 \, dt &+ \sup_{0\leq t \leq T} A[\textbf{u}_m(t), \textbf{u}_m(t)] \leq A[\textbf{u}_m(0), \textbf{u}_m(0)] +  \int_0^T 2C_B^2 \Vert \textbf{u}_m(t) \Vert_{\textbf{V}}^2 \,  dt +  \int_0^T 2 \Vert \textbf{f}(t) \Vert_{\textbf{L}^2}^2 \, dt \\
    &\leq \alpha \Vert \textbf{u}_m(0) \Vert_{\textbf{V}}^2 + 2C_B^2 [TC_1(T)+C_2(T)] + 2 \int_0^T \Vert \textbf{f} (t) \Vert_{\textbf{L}^2}^2 \, dt \\
    &\leq \alpha \Vert \textbf{u}_0 \Vert_{\textbf{V}}^2 + 2C_B^2 [TC_1(T)+C_2(T)] + 2 \int_0^T \Vert \textbf{f} (t) \Vert_{\textbf{L}^2}^2 \, dt \\
    &\leq C_3 \left(  \Vert \textbf{u}_0 \Vert_{\textbf{V}}^2 + \Vert \textbf{f} \Vert_{ L^2(0, T; \textbf{L}^2)    }^2  \right)
\end{align*}
\end{Small}
for some constant $C_3$ depending only on $T$ and $\alpha$ by the way we defined $C_1(T)$ and $C_2(T)$ in \eqref{maxum} and \eqref{C2}.

Therefore, since 
\begin{equation*}
    \sup_{0\leq t \leq T} \Vert \nabla \textbf{u}_m \Vert_{\textbf{L}^2}^2 \leq \frac{C_3}{\alpha} \left(  \Vert \textbf{u}_0 \Vert_{\textbf{V}}^2 + \Vert \textbf{f} \Vert_{ L^2(0, T; \textbf{L}^2)    }^2  \right)
\end{equation*}
and so combining with \eqref{maxum} we obtain
\begin{equation*}
     \sup_{0\leq t \leq T} \Vert \textbf{u}_m \Vert_{{\textbf{V}}}^2 \leq C_4 \left(  \Vert \textbf{u}_0 \Vert_{\textbf{V}}^2 + \Vert \textbf{f} \Vert_{ L^2(0, T; \textbf{L}^2)    }^2  \right)
\end{equation*}
for some constant $C_4$ depending only on $T$ and $\alpha$. Passing to limits, we find that $\textbf{u}\in L^{\infty}(0, T; {\textbf{V}})$ amd $\textbf{u}' \in L^2(0, T; \textbf{L}^2)$ with the stated bounds. 

As a result, for a.e. $0\leq t \leq T$ we have the identity
\begin{equation*}
    ( \textbf{u}', \textbf{v} ) + A[\textbf{u}, \textbf{v}] + B[\textbf{u}, \textbf{u}, \textbf{v}]=(\textbf{f}, \textbf{v})
\end{equation*}
for each $\textbf{v}\in {\textbf{V}}$. Defining
\begin{equation*}
    \textbf{h}\coloneqq \textbf{f}-\textbf{u}'- \left[\left(\frac{c\textbf{u}}{ \sqrt{c^2 + |\textbf{u}|^2}    }\cdot \nabla\right)\textbf{u} \right] \in \textbf{L}^2
\end{equation*}
we have that $\textbf{u}$ satisfies the elliptic problem
\begin{equation}
    A[\textbf{u}, \textbf{v}]=(\textbf{h}, \textbf{v})
\end{equation}
for a.e. $0\leq t\leq T$. Therefore, by elliptic regularity, there exists some constant $C_{\Omega}$ depending only on the geometry of $\Omega$ such that
\begin{align*}
    \Vert \textbf{u} \Vert_{\textbf{H}^2}^2 &\leq C_{\Omega} \left(  \Vert \textbf{h} \Vert_{\textbf{L}^2}^2 + \Vert \textbf{u} \Vert_{\textbf{L}^2}^2\right) \\
    &\leq 4C_{\Omega} \left(  \Vert \textbf{f} \Vert_{\textbf{L}^2}^2 + \Vert \textbf{u}' \Vert_{\textbf{L}^2}^2 + c^2 \Vert \textbf{u} \Vert_{{\textbf{V}}}^2 \right).
\end{align*}
Integrating and applying our previous estimates we find that \eqref{EstimateSupV} holds for some constant $C$.
\end{proof}

The above estimates allow us to conclude:
\begin{cor}
The solution of Proposition \ref{HigherRegularity} satisfies
\begin{equation*}
    \textbf{u}\in C([0, T]; {\textbf{V}})
\end{equation*}
after potentially being redefined on a set of measure $0$.
\end{cor}

\begin{proof}
This follows from the estimates of Proposition \ref{HigherRegularity} and Theorem 4 in Section 5.9 of \cite{evans}.
\end{proof}

\begin{cor} \label{SobolevCorollary}
The solution of Proposition \ref{HigherRegularity} satisfies
\begin{equation*}
    \textbf{u}\in L^2(0, T; C^{0, 1/2}(\overbar{\Omega}))
\end{equation*}
after potentially being redefined on a set of measure $0$ with the estimate
\begin{equation}
    \Vert \textbf{u} \Vert_{L^2(0, T; C^{0, \, 1/2})}^2 \leq C_S^2 \Vert \textbf{u} \Vert_{L^2(0, T; \textbf{H}^2)}^2. 
\end{equation}
\end{cor}

\begin{proof}
We have that for almost every $0\leq t \leq T$ $\textbf{u}(t) \in \textbf{H}^2$. Applying the Sobolev embedding theorem for $n=3$, $p=2$, and $k=2$ we obtain
\begin{equation*}
    \Vert \textbf{u} \Vert_{C^{0, \, 1/2}(\overbar{\Omega})}^2 \leq C_S^2 \Vert \textbf{u} \Vert_{\textbf{H}^2}^2
\end{equation*}
for some constant $C_S$ (where $S$ stands for Sobolev) depending only on $k, p, n$ and $\Omega$, after potentially redefining $\textbf{u}$ on a set of measure $0$. Integrating with respect to time completes the corollary.
\end{proof}

Next, we prove that our solution is in fact unique.

\subsection{Uniqueness of the Solution} The proof of the uniqueness of the strong solution is standard, with just some minor modifications for our particular model.

\begin{proposition} \label{PropositionUniqueness}
The solution of Proposition \ref{HigherRegularity} is in fact unique
\end{proposition}

\begin{proof}

Let $\textbf{u}_1$ and $\textbf{u}_2$ both be solutions of \eqref{weakForm} with initial condition \eqref{weakFormU}. Defining
\begin{equation*}
    \textbf{U}\coloneqq \textbf{u}_2 - \textbf{u}_1
\end{equation*}
we obtain, by subtracting the equations satisfied by $\textbf{u}_2$ and $\textbf{u}_1$, that $\textbf{U}$ satisfies the weak problem
\begin{align*}
     \frac{d}{dt} (\textbf{U}, \textbf{v}) + A[\textbf{U}, \textbf{v}] + B[{\textbf{u}}_2, \textbf{u}_2, \textbf{v}] - B[{\textbf{u}}_1, \textbf{u}_1, \textbf{v}]&=0 \\
\textbf{U}(0)&=0. 
\end{align*}
By using linearity in the second argument, we can write
\begin{align*}
    B[{\textbf{u}}_2, \textbf{u}_2, \textbf{v}] - B[{\textbf{u}}_1, \textbf{u}_1, \textbf{v}] &= B[{\textbf{u}}_2, \textbf{u}_2, \textbf{v}] - B[{\textbf{u}}_2, \textbf{u}_1, \textbf{v}] + B[{\textbf{u}}_2, \textbf{u}_1, \textbf{v}] -B[{\textbf{u}}_1, \textbf{u}_1, \textbf{v}] \\
    &=B[{\textbf{u}}_2, \textbf{U}, \textbf{v}]  + B[{\textbf{u}}_2, \textbf{u}_1, \textbf{v}] -B[{\textbf{u}}_1, \textbf{u}_1, \textbf{v}].
    \end{align*}
and thus we can write
\begin{align*}
     \frac{d}{dt} (\textbf{U}, \textbf{v}) + A[\textbf{U}, \textbf{v}]    = -B[{\textbf{u}}_2, \textbf{U}, \textbf{v}] + B[{\textbf{u}}_1, \textbf{u}_1, \textbf{v}] - B[{\textbf{u}}_2, \textbf{u}_1, \textbf{v}]& \\
\textbf{U}(0)=0&. 
\end{align*}
Now, letting
\begin{equation*}
    \textbf{F}=\left(\left[ \frac{c{\textbf{u}}_1}{ \sqrt{c^2 + |{\textbf{u}}_1|^2}    }  - \frac{c{\textbf{u}}_2}{ \sqrt{c^2 + |{\textbf{u}}_2|^2}    } \right]  \cdot \nabla\right)\textbf{u}_1 
\end{equation*}
we see that $\textbf{F}\in \textbf{L}^2$ and we can write 
\begin{align*}
    \frac{d}{dt} (\textbf{U}, \textbf{v}) + A[\textbf{U}, \textbf{v}]   &=-B[{\textbf{u}}_2, \textbf{U}, \textbf{v}]+ (\textbf{F}, \textbf{v})& \\
\textbf{U}(0)&=0. 
\end{align*}
Letting $\textbf{v}=\textbf{U}$ we obtain
\begin{align} \label{EquationU}
    \frac{d}{dt} (\textbf{U}, \textbf{U})+ A[\textbf{U}, \textbf{U}]   &=-B[{\textbf{u}}_2, \textbf{U}, \textbf{U}]+ (\textbf{F}, \textbf{U})& 
\end{align}
Now, we analyze 
\begin{align*}
    \left\vert B[{\textbf{u}}_2, \textbf{U}, \textbf{U}]   \right\vert \leq C_B \Vert \nabla \textbf{U} \Vert \Vert \textbf{U} \Vert \leq C_B\varepsilon \Vert \nabla \textbf{U} \Vert^2 + \frac{C_B}{4\varepsilon} \Vert \textbf{U} \Vert^2
\end{align*}
by \eqref{BBound} and we take $\varepsilon=\alpha/2C_B$ so that
\begin{align*}
     \left\vert B[{\textbf{u}}_2, \textbf{U}, \textbf{U}]   \right\vert \leq  \frac{\alpha}{2}\Vert \nabla \textbf{U} \Vert^2 + \frac{C_B^2}{2\alpha } \Vert \textbf{U} \Vert^2
\end{align*}
Next, we analyze $(\textbf{F}, \textbf{U})$.

The key is that using some simple algebra, which is relegated to Appendix \ref{AppendixA}, we can estimate
\begin{equation} \label{AppendixInequality}
    \left\vert \frac{c{\textbf{u}}_1}{ \sqrt{c^2 + |{\textbf{u}}_1|^2}    }  - \frac{c{\textbf{u}}_2}{ \sqrt{c^2 + |{\textbf{u}}_2|^2}    } \right\vert \leq \frac{C_n}{c} | {\textbf{u}}_2 - {\textbf{u}}_1 |=\frac{C_n}{c}|{\textbf{U}}|
\end{equation}
where the constant $C_n$ depends only on the dimension. In our case where $n=3$ we can take $C_3=12$. 

Therefore
\begin{align*}
    \left\vert (\textbf{F}, \textbf{U}) \right\vert &\leq \int_{U} \left\vert \left(\left[ \frac{c{\textbf{u}}_1}{ \sqrt{c^2 + |{\textbf{u}}_1|^2}    }  - \frac{c{\textbf{u}}_2}{ \sqrt{c^2 + |{\textbf{u}}_2|^2}    } \right]  \cdot \nabla\right)\textbf{u}_1  \right\vert \left\vert \textbf{U} \right\vert \, dx \\
    &\leq  \frac{12}{c} \int_\Omega  |\nabla \textbf{u}_1|  |\textbf{U}|^2 \, dx \\
    & \leq \frac{12}{c} \Vert \nabla \textbf{u}_1 \Vert_{\textbf{L}^2} \Vert \textbf{U} \Vert_{\textbf{L}^4}^2 \\
    &\leq \frac{12}{c} \Vert \nabla \textbf{u}_1 \Vert_{\textbf{L}^2} \left( \Vert \textbf{U} \Vert_{\textbf{L}^2}^{1/4} \Vert \nabla \textbf{U} \Vert_{\textbf{L}^2}^{3/4} \right)^2
    \\
    &\leq \frac{12}{c} \Vert \nabla \textbf{u}_1 \Vert_{\textbf{L}^2}  \Vert \textbf{U} \Vert_{\textbf{L}^2}^{1/2} \Vert \nabla \textbf{U} \Vert_{\textbf{L}^2}^{3/2}.
\end{align*}
Next, we apply Young's inequality with $\varepsilon$ for $p=4/3$ and $q=4$ and choose the $\varepsilon$  so that
\begin{equation*}
    \frac{12}{c}\varepsilon = \frac{\alpha}{2}
\end{equation*}
which then gives
\begin{align*}
     \left\vert (\textbf{F}, \textbf{U}) \right\vert \leq \frac{\alpha}{2} \Vert \nabla \textbf{U} \Vert_{\textbf{L}^2}^2 + C_{\varepsilon} \Vert \nabla \textbf{u}_1 \Vert_{\textbf{L}^2}^4 \Vert \textbf{U} \Vert_{\textbf{L}^2}^2
\end{align*}
and therefore using \eqref{EquationU} we obtain
\begin{small}
\begin{align*}
    \frac{1}{2}\frac{d}{dt}\Vert \textbf{U}(t) \Vert_{\textbf{L}^2}^2 + \alpha \Vert \nabla \textbf{U}(t)  \Vert_{\textbf{L}^2}^2 &\leq \frac{\alpha}{2}\Vert \nabla \textbf{U}(t) \Vert^2 + \frac{C_B^2}{2\alpha } \Vert \textbf{U}(t) \Vert^2 + \frac{\alpha}{2} \Vert \nabla \textbf{U}(t) \Vert_{\textbf{L}^2}^2 + C_{\varepsilon} \Vert \nabla \textbf{u}_1(t) \Vert_{\textbf{L}^2}^4 \Vert \textbf{U}(t) \Vert_{\textbf{L}^2}^2 \\
    &=  \alpha \Vert \nabla \textbf{U}(t)  \Vert_{\textbf{L}^2}^2 + \left( \frac{C_B^2}{2\alpha} + C_\varepsilon \Vert \nabla \textbf{u}_1(t) \Vert_{\textbf{L}^2}^4   \right) \Vert \textbf{U}(t) \Vert_{\textbf{L}^2}^2.
\end{align*}
\end{small}

Now, since $\textbf{u}_1 \in L^{\infty}(0, T; \textbf{V})$ we can use \eqref{EstimateSupV} to bound $C_\varepsilon \Vert \nabla \textbf{u}_1(t) \Vert_{\textbf{L}^2}^4 \leq K_1$ for some $K_1$ for a.e. $0\leq t \leq T$. Thus rearranging by subtracting the $ \alpha \Vert \nabla \textbf{U}(t)  \Vert_{\textbf{L}^2}^2$ term from both sides and multiplying by $2$ we can write
\begin{equation*}
    \frac{d}{dt} \Vert \textbf{U}(t) \Vert_{\textbf{L}^2}^2 \leq K \Vert \textbf{U}(t) \Vert_{\textbf{L}^2}^2 
\end{equation*}
for some $K\geq 0$ for almost every $0\leq t \leq T$. Hence by Gronwall's inequality we have
\begin{equation*}
      \Vert \textbf{U}(t)  \Vert_{\textbf{L}^2}^2 \leq e^{Kt}  \Vert  \textbf{U}(0)  \Vert_{\textbf{L}^2}^2
\end{equation*}
for $0\leq t\leq T$. But since $\textbf{U}(0)=0$ we get 
\begin{equation*}
    \Vert \textbf{U}(t)  \Vert_{\textbf{L}^2}^2 = 0
\end{equation*}
for all $0\leq t \leq T$ completing the proof of uniqueness.
\end{proof}

\subsection{Recovering the Pressure}

We would also like to make sense of the pressure term in \eqref{QRNSMOmega}. The pressure can be recovered (in the distribution) exactly as at the end of Section 3.3.5 in \cite{temam} so we omit the details. Following that construction, we find that the pressure is defined by
\begin{equation*}
    p=\frac{\partial P }{\partial t}
\end{equation*}
for some 
\begin{equation*}
    P\in C([0, T]; L^2)
\end{equation*}
understood in the distribution sense.

\section{Conclusion}

In this paper we suggested a modification to the usual Navier-Stokes model based on some intuition derived from special relativity. This modification led us to establishing the existence of unique global strong solutions for the incompressible equations in dimension $n=3$ for any $T>0$ as long as the forcing function $\textbf{f}$ was in an appropriate space. If the model proves interesting to other mathematicians, in a future paper we will investigate the higher regularity of the solutions, which should not be too difficult as the solutions should be amenable to the standard techniques.

We comment that though justifiable, the introduction of the modified term in our equations could be considered somewhat ad-hoc, though no more so than in \cite{JaraczLee}. As mentioned, in that paper, and others, the modifications are introduced by considering Brownian motion, while here we introduced them as a consequence of special and general relativity. In a future paper we will study what happens when we derive the equations by directly using some spacial relativistic principles. One possibility is replacing the mass term in Newton's law by the special relativistic mass and studying the consequences of this. This might lead to a model which should be more intuitively palatable for some. However, the fact that incompressibility has no real place in special relativity might produce interesting consequences.

\appendix  

\section{Proof of Inequality \ref{AppendixInequality}}\label{AppendixA}

We consider the difference
\begin{equation*}
    \frac{ \tilde{\textbf{y}} }{\sqrt{ c^2+|\tilde{\textbf{y}}|^2  }} - \frac{ \textbf{y} }{\sqrt{ c^2+|\textbf{y}|^2  }}
\end{equation*}
by analyzing this vector component wise. In general, this could be an $n$-vector, though of course we are interested in the case $n=3$.
Notice that
\begin{align*}
    \frac{ \tilde{y}_j }{\sqrt{ c^2+|\tilde{\textbf{y}}|^2  }} - \frac{ y_j }{\sqrt{ c^2+|\textbf{y}|^2  }}&=  \frac{ \tilde{y}_j }{\sqrt{ c^2+|\tilde{\textbf{y}}|^2  }}-\frac{ y_j }{\sqrt{ c^2+|\tilde{\textbf{y}}|^2  }} +\frac{ y_j }{\sqrt{ c^2+|\tilde{\textbf{y}}|^2  }} - \frac{ y_j }{\sqrt{ c^2+|\textbf{y}|^2  }}\\
    &= \frac{ \tilde{y}_j - y_j }{\sqrt{ c^2+|\tilde{\textbf{y}}|^2  }} + y_j \left(  \frac{ 1 }{\sqrt{ c^2+|\tilde{\textbf{y}}|^2  }} - \frac{ 1 }{\sqrt{ c^2+|\textbf{y}|^2  }} \right).
\end{align*}
Next we see
\begin{align*}
 \frac{ 1 }{\sqrt{ c^2+|\tilde{\textbf{y}}|^2  }} - \frac{ 1 }{\sqrt{ c^2+|\textbf{y}|^2  }} & = \left( \frac{  \sqrt{ c^2+|\textbf{y}|^2  } -  \sqrt{ c^2+|\tilde{\textbf{y}}|^2  }    }{  \sqrt{ c^2+|\tilde{\textbf{y}}|^2  }  \sqrt{ c^2+|\textbf{y}|^2  }    } \right) \frac{   \sqrt{ c^2+|\textbf{y}|^2  } +  \sqrt{ c^2+|\tilde{\textbf{y}}|^2  } }{  \sqrt{ c^2+|\textbf{y}|^2  } +  \sqrt{ c^2+|\tilde{\textbf{y}}|^2  } } \\
 &= \frac{|\textbf{y}|^2 - |\tilde{\textbf{y}}|^2}{ \sqrt{ c^2+|\tilde{\textbf{y}}|^2  }  \sqrt{ c^2+|\textbf{y}|^2  }    ( \sqrt{ c^2+|\textbf{y}|^2  } +  \sqrt{ c^2+|\tilde{\textbf{y}}|^2  })     } \\
  &= \frac{   \sum_{k=1}^n (y_k+\tilde{y}_k)(y_k - \tilde{y}_k )       }{ \sqrt{ c^2+|\tilde{\textbf{y}}|^2  }  \sqrt{ c^2+|\textbf{y}|^2  }    ( \sqrt{ c^2+|\textbf{y}|^2  } +  \sqrt{ c^2+|\tilde{\textbf{y}}|^2  })     }
\end{align*}
and therefore 
\begin{small}
\begin{align*}
    \left \vert \frac{ \tilde{y}_j }{\sqrt{ c^2+|\tilde{\textbf{y}}|^2  }} - \frac{ y_j }{\sqrt{ c^2+|\textbf{y}|^2  }} \right \vert &\leq \left\vert \frac{ \tilde{y}_j - y_j }{\sqrt{ c^2+|\tilde{\textbf{y}}|^2  }} \right\vert + |y_j| \frac{   \sum_{k=1}^n (|y_k|+|\tilde{y}_k|)|y_k - \tilde{y}_k |       }{ \sqrt{ c^2+|\tilde{\textbf{y}}|^2  }  \sqrt{ c^2+|\textbf{y}|^2  }    ( \sqrt{ c^2+|\textbf{y}|^2  } +  \sqrt{ c^2+|\tilde{\textbf{y}}|^2  })     } \\
    &\leq \frac{n+1}{c}|\tilde{\textbf{y}}-\textbf{y}|
\end{align*}
\end{small}
and so 
\begin{equation*}
    \left\vert \frac{ \tilde{\textbf{y}} }{\sqrt{ c^2+|\tilde{\textbf{y}}|^2  }} - \frac{ \textbf{y} }{\sqrt{ c^2+|\textbf{y}|^2  }} \right\vert \leq \frac{n(n+1)}{c} \left\vert \tilde{\textbf{y}}-\textbf{y} \right\vert
\end{equation*}
and so we can take $C_n=n(n+1)$ and $C_3=12$.

\bibliographystyle{model1-num-names}
\bibliography{ref_new}

\end{document}